\documentclass{amsart}
\usepackage{amsmath}
\usepackage{amssymb}
\usepackage{amsthm}
\usepackage[msc-links,abbrev]{amsrefs}
\usepackage{hyperref}
\usepackage[noabbrev,capitalize]{cleveref}
\usepackage{mathrsfs}
\usepackage{mathtools}
\usepackage{slashed}

\DeclareMathOperator{\tr}{tr}

\DeclareMathOperator{\dvol}{dvol}

\DeclareMathOperator{\Ric}{Ric}

\DeclareMathOperator{\Rm}{Rm}

\newcommand{\multinom}[3]{\binom{#1}{#2 \, ; \, #3}}

\newcommand{\defn}[1]{{\boldmath\bfseries#1}}

\newcommand{\cg}{\widetilde{g}}

\newcommand{\cf}{\widetilde{f}}

\newcommand{\cu}{\widetilde{u}}
\newcommand{\cv}{\widetilde{v}}

\newcommand{\cD}{\widetilde{D}}

\newcommand{\cI}{\widetilde{I}}

\newcommand{\cnabla}{\widetilde{\nabla}}

\newcommand{\cDelta}{\widetilde{\Delta}}
\newcommand{\cmE}{\widetilde{\mathcal{E}}}

\newcommand{\cmG}{\widetilde{\mathcal{G}}}

\newcommand{\cmM}{\widetilde{\mathcal{M}}}

\DeclareMathOperator{\cRm}{\widetilde{\Rm}}

\newcommand{\bbg}{\boldsymbol{g}}

\newcommand{\lv}{\lvert}
\newcommand{\rv}{\rvert}

\newcommand{\mC}{\mathcal{C}}

\newcommand{\mE}{\mathcal{E}}

\newcommand{\mG}{\mathcal{G}}

\newcommand{\mI}{\mathcal{I}}

\newcommand{\mL}{\mathcal{L}}
\newcommand{\mM}{\mathcal{M}}

\newcommand{\mO}{\mathcal{O}}

\newcommand{\mR}{\mathcal{R}}

\newcommand{\bN}{\mathbb{N}}

\newcommand{\bR}{\mathbb{R}}

\def\sideremark#1{\ifvmode\leavevmode\fi\vadjust{\vbox to0pt{\vss
 \hbox to 0pt{\hskip\hsize\hskip1em
 \vbox{\hsize3cm\tiny\raggedright\pretolerance10000
 \noindent #1\hfill}\hss}\vbox to8pt{\vfil}\vss}}}

\newcommand{\suchthatcolon}{\mathrel{}:\mathrel{}}

\newtheorem{theorem}{Theorem}[section]

\newtheorem{lemma}[theorem]{Lemma}
\newtheorem{corollary}[theorem]{Corollary}

\theoremstyle{definition}

\theoremstyle{remark}
\newtheorem{remark}[theorem]{Remark}

\numberwithin{equation}{section}

\makeatletter
\@namedef{subjclassname@2020}{%
  \textup{2020} Mathematics Subject Classification}
\makeatother

\begin{document}

\title{Curved versions of the Ovsienko--Redou operators}
\author{Jeffrey S. Case}
\address{Department of Mathematics \\ Penn State University \\ University Park, PA 16802 \\ USA}
\email{jscase@psu.edu}
\author{Yueh-Ju Lin}
\address{Department of Mathematics, Statistics, and Physics \\ Wichita State University \\ Wichita, KS 67260 \\ USA}
\email{yueh-ju.lin@wichita.edu}
\author{Wei Yuan}
\address{Department of Mathematics \\ Sun Yat-sen University \\ Guangzhou, Guangdong 510275 \\ China}
\email{yuanw9@mail.sysu.edu.cn}
\begin{abstract}
 We give a complete classification of tangential bidifferential operators of total order at most $n$ which are expressed purely in terms of the Laplacian on the ambient space of an $n$-dimensional manifold.
 This gives a curved analogue of the classification, due to Ovsienko--Redou and Clerc, of conformally invariant bidifferential operators on the sphere.
 As an application, we construct a large class of formally self-adjoint conformally invariant differential operators.
\end{abstract}
\maketitle

\section{Introduction}
\label{sec:intro}

Representation theory completely classifies~\cite{Branson1995} the conformally invariant differential operators on the $n$-sphere:
The space of conformally invariant differential operators of order $2k \in \bN$ is the span of the restriction $\cDelta^k\cu\rv_{\mG}$, where $\cDelta$ is the Laplacian on Minkowski space $(\bR^{n+1,1},-d\tau^2+dx^2)$ with $S^n$ identified with the projectivization of the null cone $\mG := \left\{ (x,\tau) \in \bR^{n+1,1} \suchthatcolon \lv x\rv^2 = \tau^2 \right\}$, and $\cu$ is an element of the space $\cmE\bigl[-\frac{n-2k}{2}\bigr]$ of functions on $\bR^{n+1,1}$ which are homogeneous of degree $-\frac{n-2k}{2}$ with respect to the dilations $(x,\tau) \mapsto (cx,c\tau)$.
Graham, Jenne, Mason and Sparling proved~\cite{GJMS1992} that when $k \leq n/2$, applying this construction to the Fefferman--Graham ambient space~\cite{FeffermanGraham2012} of a pseudo-Riemannian $n$-manifold yields nontrivial operators.
The resulting \defn{GJMS operators} are conformally invariant differential operators with leading-order term $\Delta^k$.
While $\cDelta$ is formally self-adjoint on $\bR^{n+1,1}$, an additional argument is needed to conclude that the GJMS operators are formally self-adjoint~\cites{GrahamZworski2003,FeffermanGraham2013,Juhl2013,FeffermanGraham2002}.

Representation theory also gives~\cites{OvsienkoRedou2003,Clerc2016,Clerc2017} a complete classification of the space of conformally invariant bidifferential operators on the $n$-sphere, though now the situation is much more complicated:
Ovsienko and Redou showed~\cite{OvsienkoRedou2003} that for generic $w_1,w_2 \in \bR$, the space of conformally invariant bidifferential operators $D_{2k;w_1,w_2} \colon \cmE[w_1] \otimes \cmE[w_2] \to \cmE[w_1+w_2-2k]$ of total order $2k$ is one-dimensional.
Clerc classified~\cites{Clerc2016,Clerc2017} the space of conformally invariant bidifferential operators $D_{2k;w_1,w_2}$ for the remaining weights, showing that it can be one-, two-, or three-dimensional, depending on the choice of $w_1,w_2$, the total order $2k$, and the dimension $n$.
On curved manifolds, Case, Lin and Yuan~\cite{CaseLinYuan2018b} constructed the operator $D_{2k;-\frac{n-2k}{3},-\frac{n-2k}{3}}$ on all conformal manifolds of dimension $n \geq 2k$.
Note that $w_1=w_2=-\frac{n-2k}{3}$ is the unique choice of weights for which $D_{2k;w_1,w_2}$ can be \defn{formally self-adjoint};
i.e.\ for which the associated Dirichlet form
\begin{equation*}
 (u_1, u_2, u_3) \mapsto \int_M u_1 \, D_{2k;w_1,w_2}( u_2 \otimes u_3) \, \dvol
\end{equation*}
on compactly-supported functions is symmetric.
On the sphere, one can prove that $D_{2k;-\frac{n-2k}{3},-\frac{n-2k}{3}}$ is formally self-adjoint via the realization of its Dirichlet form as the residue of a symmetric conformally invariant trilinear form~\cite{BeckmannClerc2012}.
Case, Lin and Yuan~\cite{CaseLinYuan2018b} proved by direct computation that $D_{2k;-\frac{n-2k}{3},-\frac{n-2k}{3}}$ is formally self-adjoint on curved manifolds if $k \leq 2$.

Our primary goal is to show that the full range of operators identified by Ovsienko--Redou~\cite{OvsienkoRedou2003} and Clerc~\cites{Clerc2016,Clerc2017} admit curved analogues under the natural dimensional assumption $n \geq 2k$.
Our secondary goal is to better understand the formal self-adjointness of $D_{2k;-\frac{n-2k}{3},-\frac{n-2k}{3}}$ and related linear operators.

We approach this problem by adapting the construction~\cite{GJMS1992} of the GJMS operators to bidifferential operators (cf.\ \cite{CaseLinYuan2018b}).
Specifically, given a conformal manifold $(M^n,[g])$, we determine those coefficients $b_{s,t}$ for which the operator
\begin{equation}
 \label{eqn:general-form}
 \cu \otimes \cv \mapsto \sum_{s=0}^k \sum_{t=0}^{k-s} b_{s,t}\cDelta^{k-s-t}\left(\bigl( \cDelta^s\cu\bigr) \bigl(\cDelta^t\cv\bigr) \right)
\end{equation}
on the Fefferman--Graham~\cite{FeffermanGraham2012} ambient space $(\cmG,\cg)$ defines a tangential operator on $\cmE[w_1] \otimes \cmE[w_2]$.
The relevance of this is that tangential operators on $(\cmG,\cg)$ determine conformally invariant operators on $(M^n,g)$;
see \cref{sec:bg} for details.
The answer, given in \cref{when-tangential} below, is a system of two recursive relations.
This recursive relation is readily solved when $w_1,w_2 \not\in \mI_k$ and $w_1+w_2 \not\in \mO_k$, where
\begin{align*}
 \mI_k & := \left\{ -\frac{n-2k}{2} - \ell \right\}_{\ell=0}^{k-1} , \\
 \mO_k & := \left\{ -\frac{n-2k}{2} + \ell \right\}_{\ell=0}^{k-1} .
\end{align*}
Restricting this observation to the conformal class of the round sphere recovers the conformally invariant bidifferential operators classified by Ovsienko and Redou~\cite{OvsienkoRedou2003}.
Taking $w_1=w_2=-\frac{n-2k}{3}$ on general conformal $2k$-manifolds recovers the conformally invariant operators previously constructed by Case, Lin and Yuan~\cite{CaseLinYuan2018b}.

More generally, the above recursive relation can be solved for all $w_1, w_2 \in \bR$.
As in Clerc's work on the sphere~\cites{Clerc2016,Clerc2017}, the space of tangential operators~\eqref{eqn:general-form} --- called \defn{Ovsienko--Redou operators} --- can be one-, two-, or three-dimensional.
The following three theorems completely characterize the space of Ovsienko--Redou operators.

Our first theorem characterizes those weights for which the space of Ovsienko--Redou operators is one-dimensional, and gives a basis for this space.
To simplify our notation, we set
\begin{equation*}
 \multinom{k}{s}{t} := \frac{k!}{(k-s-t)!s!t!}
\end{equation*}
for $k,s,t \in \bN_0$.

\begin{theorem}
 \label{ovsienko-redou-1d}
 Let $(M^n,[g])$ be a conformal manifold.
 Let $k \leq n/2$ be a positive integer and let $w_1,w_2 \in \bR$.
 Suppose that either
 \begin{enumerate}
  \item at most one of $w_1 \in \mI_k$ or $w_2 \in \mI_k$ or $w_1+w_2 \in \mO_k$ holds;
  \item $w_1,w_2 \in \mI_k$ with $w_1+w_2+n \leq k$, but $w_1+w_2 \not\in \mO_k$;
  \item $w_1 \in \mI_k$ and $w_1+w_2 \in \mO_k$ with $w_2 \geq k$, but $w_2 \not\in \mI_k$; or
  \item $w_2 \in \mI_k$ and $w_1+w_2 \in \mO_k$ with $w_1 \geq k$, but $w_1 \not\in \mI_k$.
 \end{enumerate}
 Then the space of Ovsienko--Redou operators is one-dimensional and spanned by
 \begin{align*}
  \cD_{2k;w_1,w_2}(\cu \otimes \cv) & := \sum_{s=0}^k \sum_{t=0}^{k-s} \multinom{k}{s}{t} a_{s,t} \cDelta^{k-s-t} \bigl( (\cDelta^s\cu)(\cDelta^t\cv) \bigr) , \\
  a_{s,t} & := \frac{\Gamma\bigl(-w_1-w_2-\frac{n-2k}{2}+s+t\bigr)\Gamma\bigl(w_1 + \frac{n}{2}-s\bigr)\Gamma\bigl(w_2 + \frac{n}{2} - t\bigr)}{\Gamma\bigl(-w_1-w_2-\frac{n-2k}{2}\bigr)\Gamma\bigl(w_1+\frac{n-2k}{2}\bigr)\Gamma\bigl(w_2+\frac{n-2k}{2}\bigr)} ,
 \end{align*}
 where the ratio of Gamma functions is understood via analytic continuation.
 Moreover, for each $g \in [g]$, the operator $\cD_{2k;w_1,w_2}$ determines a natural bidifferential operator $D_{2k;w_1,w_2}^g \colon C^\infty(M)^{\otimes 2} \to C^\infty(M)$ such that
 \begin{equation}
  \label{eqn:conformal-transformation}
  D_{2k;w_1,w_2}^{e^{2\Upsilon}g} \bigl( u \otimes v \bigr) = e^{(w_1+w_2-2k)\Upsilon} D_{2k;w_1,w_2}^{g} \bigl( e^{-w_1\Upsilon}u \otimes e^{-w_2\Upsilon}v \bigr)
 \end{equation}
 for all $u,v,\Upsilon \in C^\infty(M)$.
\end{theorem}

The operator $D_{2k;w_1,w_2}^g$ is \defn{natural} in the sense that it can be expressed as a complete contraction of polynomials in the covariant derivatives of the Riemannian curvature tensor, of $u$, and of $v$.
The conformal transformation law~\eqref{eqn:conformal-transformation} is succinctly expressed as the statement that $D_{2k;w_1,w_2}$ is a conformally invariant map
\begin{equation*}
 D_{2k;w_1,w_2} \colon \mE[w_1] \otimes \mE[w_2] \to \mE[w_1 + w_2 - 2k] ,
\end{equation*}
where $\mE[w]$ is the conformal density bundle
\begin{equation*}
 \mE[w] := \left\{ \cu \rv_{\mG} \suchthatcolon \cu \in \cmG[w] \right\} ;
\end{equation*}
see \cref{sec:bg} for further details.

Recall that the Gamma function $\Gamma(z)$ has no zeros, has simple poles at the nonpositive integers, and that $\Gamma(z)$ is real-valued for $z \in \bR$.
Under the conditions of \cref{ovsienko-redou-1d}, any pole of the numerator of $a_{s,t}$ is also a pole of the denominator.
Moreover, the order of each pole of the numerator is at most the order of the same pole of the denominator, and there is at least one choice of $s,t \in \bN_0$ such that $a_{s,t} \not= 0$.
It follows that $D_{2k;w_1,w_2}$ is a nontrivial real operator under the conditions of \cref{ovsienko-redou-1d}.

The exceptional set $\mI_k$ in \cref{ovsienko-redou-1d} has an interesting interpretation in terms of harmonic extensions.
Specifically, if $w \not\in \mI_k$ and $u \in \mE[w]$, then there is~\cite{GJMS1992} an extension $\cu \in \cmE[w]$ of $u$ such that $\cDelta^k\cu = 0$ along the set $\mG \subset \cmG$ which is identified with $(M^n,[g])$.
By contrast, if $w_1 \in \mI_k$ and $w_1,w_2$ satisfy one of the assumptions in \cref{ovsienko-redou-1d}, then $a_{s,t}=0$ if $s < k-i$, where $w_1 = -\frac{n-2k}{2}-i$.
Thus the only nonzero terms in the expression for $\cD_{2k;w_1,w_2}$ involve those powers $\cDelta^su$ which cannot be assumed to vanish along $\mG$.

Our second theorem characterizes those weights for which the space of Ovsienko--Redou operators is two-dimensional, and gives a basis for this space.

\begin{theorem}
 \label{ovsienko-redou-2d}
 Let $(M^n,[g])$ be a conformal manifold.
 Let $k \leq n/2$ be a positive integer and let $w_1,w_2 \in \bR$.
 \begin{enumerate}
  \item If $w_1,w_2 \in \mI_k$ with $w_1+w_2+n > k$ but $w_1+w_2 \not\in \mO_k$, then the space of Ovsienko--Redou operators is two-dimensional and spanned by
  \begin{align*}
   D_{2k;w_1,w_2}^{(1)} & := \sum_{s=k-i}^k \sum_{t=0}^{k-s} \multinom{k}{s}{t} a_{s,t} \cDelta^{k-s-t} \bigl( (\cDelta^s\cu) (\cDelta^t\cv) \bigr) , \\
   D_{2k;w_1,w_2}^{(2)} & := \sum_{s=0}^j \sum_{t=k-j}^{k-s} \multinom{k}{s}{t} a_{s,t} \cDelta^{k-s-t} \bigl( (\cDelta^s\cu) (\cDelta^t\cv) \bigr) , \\
   a_{s,t} & := \frac{\Gamma\bigl(-w_1-w_2-\frac{n-2k}{2}+s+t\bigr)\Gamma\bigl(w_1 + \frac{n}{2}-s\bigr)\Gamma\bigl(w_2 + \frac{n}{2} - t\bigr)}{\Gamma\bigl(-w_1-w_2-\frac{n-2k}{2}\bigr)\Gamma\bigl(w_1+\frac{n-2k}{2}\bigr)\Gamma\bigl(w_2+\frac{n}{2}\bigr)} ,
  \end{align*}
  where $w_1 = -\frac{n-2k}{2}-i$ and $w_2 = -\frac{n-2k}{2}-j$ and the ratio of Gamma functions is understood via analytic continuation.
  \item If $w_1 \in \mI_k$ and $w_1+w_2 \in \mO_k$ with $w_2 < k$ but $w_2 \not\in \mI_k$, then the space of Ovsienko--Redou operators is two-dimensional and spanned by
  \begin{align*}
   D_{2k;w_1,w_2}^{(1)} & := \sum_{s=k-i}^k \sum_{t=0}^{k-s} \multinom{k}{s}{t} a_{s,t} \cDelta^{k-s-t} \bigl( (\cDelta^s\cu) (\cDelta^t\cv) \bigr) , \\
   D_{2k;w_1,w_2}^{(2)} & := \sum_{s=0}^{j} \sum_{t=0}^{j-s} \multinom{k}{s}{t} a_{s,t} \cDelta^{k-s-t} \bigl( (\cDelta^s\cu) (\cDelta^t\cv) \bigr) , \\
   a_{s,t} & := \frac{\Gamma\bigl(-w_1-w_2-\frac{n-2k}{2}+s+t\bigr)\Gamma\bigl(w_1 + \frac{n}{2}-s\bigr)\Gamma\bigl(w_2 + \frac{n}{2} - t\bigr)}{\Gamma\bigl(-w_1-w_2-\frac{n-2k}{2}\bigr)\Gamma\bigl(w_1+\frac{n}{2}\bigr)\Gamma\bigl(w_2+\frac{n-2k}{2}\bigr)} ,
  \end{align*}
  where $w_1 = -\frac{n-2k}{2}-i$ and $w_1+w_2 = -\frac{n-2k}{2} + j$ and the ratio of Gamma functions is understood via analytic continuation.
  \item If $w_2 \in \mI_k$ and $w_1+w_2 \in \mO_k$ with $w_1 < k$ but $w_1 \not\in \mI_k$, then the space of Ovsienko--Redou operators is two-dimensional and spanned by
  \begin{align*}
   D_{2k;w_1,w_2}^{(1)} & := \sum_{s=0}^i \sum_{t=k-i}^{k-s} \multinom{k}{s}{t} a_{s,t} \cDelta^{k-s-t} \bigl( (\cDelta^s\cu) (\cDelta^t\cv) \bigr) , \\
   D_{2k;w_1,w_2}^{(2)} & := \sum_{s=0}^j \sum_{t=0}^{j-s} \multinom{k}{s}{t} a_{s,t} \cDelta^{k-s-t} \bigl( (\cDelta^s\cu) (\cDelta^t\cv) \bigr) , \\
   a_{s,t} & := \frac{\Gamma\bigl(-w_1-w_2-\frac{n-2k}{2}+s+t\bigr)\Gamma\bigl(w_1 + \frac{n}{2}-s\bigr)\Gamma\bigl(w_2 + \frac{n}{2} - t\bigr)}{\Gamma\bigl(-w_1-w_2-\frac{n-2k}{2}\bigr)\Gamma\bigl(w_1+\frac{n-2k}{2}\bigr)\Gamma\bigl(w_2+\frac{n}{2}\bigr)} ,
  \end{align*}
  where $w_2 = -\frac{n-2k}{2}-i$ and $w_1+w_2 = -\frac{n-2k}{2} + j$ and the ratio of Gamma functions is understood via analytic continuation.
 \end{enumerate}
 Moreover, the operators $\cD_{2k;w_1,w_2}^{(j)}$, $j \in \{ 1, 2\}$, determine natural conformally invariant bidifferential operators $D_{2k;w_1,w_2}^{(j)} \colon \mE[w_1] \otimes \mE[w_2] \to \mE[w_1 + w_2 - 2k]$.
\end{theorem}

The only difference in the coefficients $a_{s,t}$ in \cref{ovsienko-redou-1d} and \cref{ovsienko-redou-2d} is a change in the denominators which ensures that $a_{s,t}$ are real-valued and not all zero in the range $s,t \geq0$ and $s+t \leq k$.
The summations in \cref{ovsienko-redou-2d} are indexed so that only nonzero terms $a_{s,t}$ are included.

Our third theorem characterizes those weights for which the space of Ovsienko--Redou operators is three-dimensional, and gives a basis for this space.
Note that this case never occurs when $n>2k$, and there is only one choice of weights for which this case occurs when $n=2k$.

\begin{theorem}
 \label{ovsienko-redou-3d}
 Let $(M^n,[g])$ be a conformal manifold.
 Let $k \leq n/2$ be a positive integer and let $w_1,w_2 \in \bR$.
 Suppose that $w_1,w_2 \in \mI_k$ and $w_1+w_2 \in \mO_k$.
 Then $w_1=w_2=0$ and $k=n/2$, and the space of Ovsienko--Redou operators is three-dimensional and spanned by
 \begin{align*}
  \cD_{n;0,0}^{(1)}(\cu \otimes \cv) & := \cDelta^k(\cu\cv) , \\
  \cD_{n;0,0}^{(2)}(\cu \otimes \cv) & := \cu\cDelta^k\cv , \\
  \cD_{n;0,0}^{(3)}(\cu \otimes \cv) & := \cv\cDelta^k\cu .
 \end{align*}
 Moreover, the operators $\cD_{2k;w_1,w_2}^{(j)}$, $j \in \{ 1, 2, 3 \}$, determine natural conformally invariant bidifferential operators $D_{2k;w_1,w_2}^{(j)} \colon \mE[w_1] \otimes \mE[w_2] \to \mE[w_1 + w_2 - 2k]$.
\end{theorem}

The restriction $n \geq 2k$ is only imposed in \cref{ovsienko-redou-1d,ovsienko-redou-2d,ovsienko-redou-3d} to ensure that the induced operator $D_{2k;w_1,w_2}$ is independent of the ambiguities of the ambient metric~\cite{FeffermanGraham2012}.
In particular, if $n$ is odd or if $(M^n,[g])$ is either locally conformally flat or admits an Einstein metric, then the operators $D_{2k;w_1,w_2}$ can be defined for all $k \in \bN$.
We leave the details of these cases, including the determination of the dimension of the space of Ovsienko--Redou operators, to the interested reader.

We define the Ovsienko--Redou operators in terms of the ambient Laplacian for ease of applications.
For example, a well-known procedure~\cites{Matsumoto2013,Takeuchi2017} involving a specific choice of extension yields intrinsic formulas for the Ovsienko--Redou operators on Einstein manifolds.
To simplify the exposition, we present the formula only in the case of weights as in \cref{ovsienko-redou-1d}.

\begin{theorem}
 \label{einstein}
 Let $(M^n,g)$ be an Einstein manifold with $\Ric = 2(n-1)\lambda g$.
 Let $k \in \bN$ and let $w_1,w_2 \in \bR$ satisfy one of the conditions of \cref{ovsienko-redou-1d}.
 Then
 \begin{equation*}
  D_{2k;w_1,w_2}(u \otimes v) = \sum_{s=0}^k \sum_{t=0}^{k-s} \multinom{k}{s}{t} a_{s,t} L_{k-s-t;w_1+w_2-2s-2t} \left( L_{s;w_1}(u) L_{t;w_2}(v) \right) ,
 \end{equation*}
 where
 \begin{equation*}
  L_{r;w} := \prod_{j=0}^{r-1} \bigl( \Delta + 2(w-2j)(n+w-2j-1)\lambda \bigr) .
 \end{equation*}
\end{theorem}

\Cref{einstein} gives an alternative expression for the Ovsienko--Redou operators on the sphere (cf.\ \cites{OvsienkoRedou2003,BeckmannClerc2012}). 
The proof of \cref{einstein} is readily adapted to handle the case of weights as in \cref{ovsienko-redou-2d} or \cref{ovsienko-redou-3d}.

\Cref{ovsienko-redou-1d} suggests that the commutator of the Ovsienko--Redou operator $D_{2k;-\frac{n-2k}{3},-\frac{n-2k}{3}}$, $n > 2k$, with first spherical harmonics is proportional to the Ovsienko--Redou operator $D_{2k-2;-\frac{n-2k+3}{3},-\frac{n-2k}{3}}$.
We verify this by direct computation.

\begin{theorem}
 \label{commutator}
 Let $(S^n,d\theta^2)$ be the round $n$-sphere and let $k \in \bN$ be such that $n > 2k$.
 Denote $D_{2k} := D_{2k;-\frac{n-2k}{3},-\frac{n-2k}{3}}$ and let $\{ x^0, \dotsc, x^n \}$ be the standard Cartesian coordinates on $\bR^{n+1} \supset S^n$.
 Then
 \begin{equation*}
  \sum_{i=0}^n x^i[D_{2k},x^i] = -\frac{k(n+2k-2)(n+k-3)(n-2k)}{18}D_{2k-2;-\frac{n-2k+3}{3},-\frac{n-2k}{3}} ,
 \end{equation*}
 where $[D,f](u\otimes v) := D\bigl((uf) \otimes v\bigr) - f\,D\bigl(u \otimes v\bigr)$ for all $u,v,f \in C^\infty(S^n)$.
\end{theorem}

The proof of \cref{commutator} proceeds by direct computation in the ambient space (cf.\ \cite{Case2019fl}).
We expect our method also computes the commutator of any Ovsienko--Redou operator with the first spherical harmonics.
We have presented only the case $w_1=w_2=-\frac{n-2k}{3}$ because of the expectation of a strong link between formally self-adjoint conformally invariant polydifferential operators, sharp Sobolev inequalities, and curvature prescription problems~\cite{CaseLinYuan2018b}.
Specifically, it is natural to ask if the sharp Sobolev inequality
\begin{equation}
 \label{eqn:sobolev}
 \int_{S^n} u\,D_{2k;-\frac{n-2k}{3},-\frac{n-2k}{3}}(u\otimes u)\,\dvol_g \geq C\left( \int_{S^n} \lv u\rv^{\frac{3n}{n-2k}}\,\dvol_g \right)^{\frac{n-2k}{n}}
\end{equation}
holds for all nonnegative $u \in C^\infty(S^n)$, and if so, to classify the minimizers.
\Cref{commutator} establishes a key step in applying the rearrangement-free strategy of Frank and Lieb~\cites{Case2019fl,FrankLieb2012b} to classify the local minimizers of~\eqref{eqn:sobolev}, if any.

In order to use variational methods to prove~\eqref{eqn:sobolev} and to study related curvature prescription problems, one needs to know that the Ovsienko--Redou operators $D_{2k;-\frac{n-2k}{3},-\frac{n-2k}{3}}$ are formally self-adjoint.
Case, Lin and Yuan~\cite{CaseLinYuan2018b} proved this in the case $k\leq2$ and conjectured it for general $k \in \bN$.
We verify formal self-adjointness for $k \leq 3$.

\begin{theorem}
 \label{or-fsa}
 Let $(M^n,g)$ be a pseudo-Riemannian manifold and let $k \leq 3$ be a positive integer.
 Then $D_{2k;-\frac{n-2k}{3},-\frac{n-2k}{3}}$ is formally self-adjoint.
\end{theorem}

We prove \cref{or-fsa} by using the definition of the Ovsienko--Redou operators in terms of ambient Laplacians to derive an explicit formula for $D_{2k;-\frac{n-2k}{3},-\frac{n-2k}{3}}$ in terms of the formally self-adjoint ``building blocks'' used by Juhl~\cite{Juhl2013} in his formulas for the GJMS operators.
We expect that an analogous formula can be derived for the Ovsienko--Redou operators (cf.\ \cites{Juhl2013,FeffermanGraham2013}).

As another application of \cref{ovsienko-redou-1d,ovsienko-redou-2d,ovsienko-redou-3d}, we construct a large family of conformally invariant differential operators mapping $\mE\bigl[-\frac{n-2k}{2}\bigl]$ to $\mE\bigl[-\frac{n+2k}{2}\bigr]$ by inserting a natural scalar Riemannian invariant into a suitably chosen Ovsienko--Redou operator.
Moreover, we show that these operators are formally self-adjoint when they are of order at most six.

\begin{theorem}
 \label{linear-operators}
 Let $(M^n,[g])$ be a conformal manifold, let $k \leq n/2$ be a positive integer, and let $\cI \in \cmE[-2\ell]$, $\ell \leq k$, be a natural scalar Riemannian invariant on $(\cmG,\cg)$.
 Then the operator $\cD \colon \cmE\bigl[-\frac{n-2k}{2}\bigr] \to \cmE\bigl[-\frac{n+2k}{2}\bigr]$,
 \begin{equation*}
  \cD(\cu) := \cD_{2k-2\ell;-2\ell,-\frac{n-2k}{2}}(\cI \otimes \cu)
 \end{equation*}
 is tangential.
 In particular, $\cD$ induces a natural conformally invariant differential operator $D \colon \mE\big[ -\frac{n-2k}{2} \bigr] \to \mE\bigl[ -\frac{n+2k}{2} \bigr]$.
 Moreover, if $k \leq \ell + 3$, then $D$ is formally self-adjoint.
\end{theorem}

The cases $\ell=0$ and $\ell=k$ recover the GJMS operators and multiplication by a scalar conformal invariant, respectively.
The next simplest case of \cref{linear-operators} is when $\ell=2$ and $k=3$.
Up to a multiplicative constant, the only nontrivial natural scalar Riemannian invariant on $(\cmG,\cg)$ which is homogeneous of weight $-4$ is $\lv \cRm\rv^2$.
In this case
\begin{equation*}
 \cD(\cu) = \frac{n-10}{2}\cDelta(\lv\cRm\rv^2\cu) + \frac{n-10}{2}\lv\cRm\rv^2\cDelta\cu + 2\cu\cDelta\lv\cRm\rv^2 ,
\end{equation*}
which induces the manifestly formally self-adjoint operator
\begin{equation*}
 Du = \frac{n-10}{2}\Delta(\lv W\rv^2u) + \frac{n-10}{2}\lv W\rv^2\Delta u + \left( 2\Delta\lv W\rv^2 - \frac{(n-6)^2}{2}J\lv W\rv^2 \right)u
\end{equation*}
on $(M^n,g)$.

More generally, \cref{ovsienko-redou-1d,ovsienko-redou-2d,ovsienko-redou-3d} produce a much larger class of conformally invariant differential operators by allowing $\cu \in \cmE[w]$ for $w$ arbitrary.
A key point in \cref{linear-operators} is that $\cD$ is formally self-adjoint in $(\cmG,\cg)$, and hence the induced operator on $(M^n,g)$ is plausibly --- and actually when $k \leq \ell + 3$ --- formally self-adjoint.

This article is organized as follows:

In \cref{sec:bg} we collect necessary background on the ambient metric.

In \cref{sec:construction} we characterize precisely when~\eqref{eqn:general-form} is tangential and thereby prove \cref{ovsienko-redou-1d,ovsienko-redou-2d,ovsienko-redou-3d}.

In \cref{sec:commutator} we prove \cref{commutator}.

In \cref{sec:einstein} we prove \cref{einstein}.

In \cref{sec:fsa} we prove \cref{linear-operators}.

\section{Background}
\label{sec:bg}

We begin by recalling some relevant aspects of ambient spaces needed to construct and study conformally invariant polydifferential operators.

Let $(M^n,[g])$ be a conformal manifold of signature $(p,q)$.
The conformal class $[g]$ is equivalent to a ray subbundle
\begin{equation*}
 \mG := \left\{ (x,g_x) \suchthatcolon x \in M, g \in [g] \right\} \subset S^2T^\ast M .
\end{equation*}
Regard $\mG$ as a principle $\bR_+$-bundle with \defn{dilation} $\delta_\lambda \colon \mG \to \mG$,
\begin{equation*}
 \delta_\lambda(x,g_x) := (x,\lambda^2g_x) ,
\end{equation*}
and projection $\pi \colon \mG \to M$,
\begin{equation*}
 \pi(x,g_x) := x .
\end{equation*}
The \defn{canonical metric} on $\mG$ is the degenerate metric $\bbg$ defined by
\begin{equation*}
 \bbg(X,Y) := g_x(\pi_\ast X,\pi_\ast Y)
\end{equation*}
for $X,Y \in T_{(x,g_x)}\mG$.
Given a choice of metric $g \in [g]$, we identify $\bR_+ \times M \cong \mG$ via $(t,x) \cong (x,t^2g_x)$.

Consider now the space $\mG \times \bR$.
We extend the projection and dilation to maps $\pi \colon \mG \times \bR \to M$ and $\delta_\lambda \colon \mG \times \bR \to \mG \times \bR$ in the natural way: $\pi\bigl(x,g_x,\rho\bigr) := x$ and $\delta_\lambda\bigl(x,g_x,\rho\bigr) := \bigl(x,\lambda^2g_x,\rho\bigr)$.
Let $\iota \colon \mG \to \mG \times \bR$ denote the inclusion $\mG \hookrightarrow \mG \times \{ 0 \}$.

A \defn{pre-ambient space} is a pair $(\cmG,\cg)$ consisting of a dilation invariant subspace $\cmG \subset \mG \times \bR$ containing $\mG \times \{ 0 \}$ and a pseudo-Riemannian metric $\cg$ of signature $(p+1,q+1)$ satisfying $\delta_\lambda^\ast\cg = \lambda^2\cg$ and $\iota^\ast\cg = \bbg$.
Let $\rho$ denote the $\bR$-coordinate in $\mG \times \bR$.
An \defn{ambient space} is a pre-ambient space $(\cmG,\cg)$ such that
\begin{enumerate}
 \item $\Ric(\cg) = O(\rho^\infty)$, if $n$ is odd;
 \item $\Ric(\cg) = O^+(\rho^{(n-2)/2})$, if $n$ is even.
\end{enumerate}
Here we say that a symmetric two-tensor field $S \in \Gamma(S^2T^\ast\cmG)$ \defn{is $O^+(\rho^m)$} if
\begin{enumerate}
 \item $S = O(\rho^m)$; and
 \item for each $z \in \mG$, the tensor $(\iota^\ast(\rho^{-m}S))(z) \in S^2T_z^\ast\mG$ is of the form $\pi^\ast s$ for some $s \in S^2T_{x}^\ast M$, $x := \pi(z)$, satisfying $\tr_{g_x}s = 0$.
\end{enumerate}
Fefferman and Graham proved~\cite{FeffermanGraham2012}*{Theorem~2.3} that ambient spaces always exist.
Indeed, given a choice of metric $g \in [g]$, there is~\cite{FeffermanGraham2012}*{Theorem~2.9(A)} a one-parameter family of metrics $g_\rho$ on $M$, the $\rho$ coefficients of which depend naturally on $g$, such that $g_0 = g$ and
\begin{equation}
 \label{eqn:straight-and-normal}
 \cg := 2\rho\,dt^2 + 2t\,dt\,d\rho + t^2 g_\rho
\end{equation}
is an ambient metric on $\cmG := \bR_+ \times M \times (-\varepsilon,\varepsilon)$ for some $\varepsilon>0$.
Moreover, $g_\rho$ is unique~\cite{FeffermanGraham2012}*{Theorem~2.9(B)} modulo $O(\rho^\infty)$ if $n$ is odd, and modulo $O^+(\rho^{n/2})$ if $n$ is even.

Let $X$ be the infinitesimal generator of the dilations $\delta_\lambda \colon \cmG \to \cmG$.
In coordinates, if $g \in [g]$ is given, then $X = t\partial_t$.
Given $w \in \bR$, define
\begin{equation*}
 \cmE[w] := \left\{ \cf \in C^\infty(\cmG) \suchthatcolon X\cf = w\cf \right\} .
\end{equation*}
The \defn{space of conformal densities of weight $w$} is
\begin{equation*}
 \mE[w] := \left\{ \cf\rv_{\mG} \suchthatcolon \cf \in \cmE[w] \right\} .
\end{equation*}
Note that if $f \in \mE[w]$, then $f(x,\lambda^2g_x) = \lambda^w f(x,g_x)$;
i.e.\ $\mE[w]$ may be identified with the space of equivalence classes for the equivalence relation
\begin{equation*}
 (f,g) \sim (e^{wu}f, e^{2u}g) .
\end{equation*}
Note also that if $\cI$ is a natural scalar Riemannian invariant which is homogeneous of degree $-2k$, $k \leq n/2$, with respect to dilations, then $\cI\rv_{\rho=0,t=1}$ determines a local conformal invariant of weight $-2k$ on $(M^n,g)$.
Two examples are~\cite{FeffermanGraham2012}
\begin{align*}
 \lv\cRm\rv^2 \rv_{\rho=0,t=1} & = \lv W\rv^2 , \\
 \cDelta\lv\cRm\rv^2 \rv_{\rho=0,t=1} & = \Delta\lv W\rv^2 + 4(n-10)\nabla^l(W_{ijkl}C^{ijk}) \\
  & \quad - 4(n-10)P_m^iW_{ijkl}W^{mjkl} + 4(n-5)(n-10)\lv C\rv^2 - 4J\lv W\rv^2 ,
\end{align*}
where $W_{ijkl}$ is the Weyl tensor and $C_{ijk} := \nabla_iP_{jk} - \nabla_jP_{ik}$ is the Cotton tensor.

Set $Q := \lv X\rv^2 \in \cmE[2]$, so that $Q$ is a defining function for $\mG \subset \cmG$.
A polydifferential operator
\begin{equation*}
 \cD \colon \cmE[w_1] \otimes \dotsm \otimes \cmE[w_j] \to \cmE[w_1 + \dotsm + w_j - 2k]
\end{equation*}
is \defn{tangential} if $\iota^\ast \bigl(\cD(\cu_1,\dotsc,\cu_j)\bigr)$ depends only on the restrictions $\iota^\ast\cu_1, \dotsc, \iota^\ast\cu_j$.
In this case $\cD$ determines a conformally invariant polydifferential operator
\begin{equation*}
 D \colon \mE[w_1] \otimes \dotsm \otimes \mE[w_j] \to \mE[ w_1 + \dotsm + w_j - 2k ]
\end{equation*}
by
\begin{equation*}
 D(u_1, \dotsc, u_j) := \iota^\ast \left( \cD( \cu_1 , \dotsc , \cu_j ) \right) ,
\end{equation*}
where $\cu_1, \dotsc, \cu_j$ are such that $\iota^\ast\cu_i = u_i$ for $i \in \{ 1 , \dotsc , j \}$.
The operator $D$ is well-defined in odd dimensions, and in even dimensions provided the restriction of $\cD$ to $\rho^{-1}(\{0\})$ depends only on $\cg$ modulo $O^+(\rho^{n/2})$.
This observation is responsible for the condition $k \leq n/2$ in \cref{ovsienko-redou-1d,ovsienko-redou-2d,ovsienko-redou-3d} (cf.\ \cite{GJMS1992}).

\section{Classification of curved Ovsienko--Redou operators}
\label{sec:construction}

We now turns to the proofs of \cref{ovsienko-redou-1d,ovsienko-redou-2d,ovsienko-redou-3d}.
Our first task is to characterize the bidifferential operators~\eqref{eqn:general-form} which are tangential.

\begin{lemma}
 \label{when-tangential}
 Fix a dimension $n \in \bN$ and weights $w_1,w_2 \in \bR$.
 Given $k \in \bN_0$, the ambient operator
 \begin{equation*}
  \cD(\cu \otimes \cv) := \sum_{s=0}^k \sum_{t=0}^{k-s} \multinom{k}{s}{t} a_{s,t}\cDelta^{k-s-t} \left( \bigl(\cDelta^s\cu\bigr) \bigl(\cDelta^t\cv\bigr) \right)
 \end{equation*}
 is tangential if and only if
 \begin{subequations}
  \label{eqn:when-tangential}
  \begin{align}
  	\label{eqn:when-tangential-s} (2w_1 + n - 2s - 2)a_{s+1,t} & = -(2w_1 + 2w_2 + n - 2k - 2s - 2t)a_{s,t} , \\
  	\label{eqn:when-tangential-t} (2w_2 + n - 2t - 2)a_{s,t+1} & = -(2w_1 + 2w_2 + n - 2k - 2s - 2t)a_{s,t}
  \end{align}
 \end{subequations}
 for all $s,t \in \bN_0$ such that $s+t \leq k$.
\end{lemma}

\begin{proof}
 It is known~\cite{GJMS1992}*{Equation~(1.8)} that
 \begin{equation}
  \label{eqn:Deltaell-Q-commutator}
  [ \cDelta^\ell , Q ] = 2\ell \cDelta^{\ell-1}(2X + n + 4 - 2\ell)
 \end{equation}
 for all $\ell \in \bN$.
 We immediately deduce that
 \begin{multline}
  \label{eqn:tangential-1}
  \cDelta^r \left( \bigl( \cDelta^s(Q\cu) \bigr) \bigl( \cDelta^t\cv \bigr) \right) \equiv 2s(2w_1 + n - 2s) \cDelta^{r} \left( \bigl( \cDelta^{s-1}\cu \bigr) \bigl( \cDelta^t\cv \bigr) \right) \\
   + 2r(2w_1 + 2w_2 + n - 2r - 4s - 4t)\cDelta^{r-1} \left( \bigl( \cDelta^s\cu \bigr) \bigl( \cDelta^t\cv \bigr) \right) \mod Q
 \end{multline}
 for all $\cu \in \cmE[w_1-2]$, all $\cv \in \cmE[w_2]$, and all $r,s,t \in \bN_0$.
 Equation~\eqref{eqn:tangential-1} implies that
 \begin{multline*}
  \frac{1}{2}\cD(Q\cu \otimes \cv) \equiv \sum_{s=0}^{k} \sum_{t=0}^{k-s} k\multinom{k-1}{s}{t} \Bigl[ (2w_1 + n - 2s - 2)a_{s+1,t} \\
   + (2w_1 + 2w_2 + n - 2k - 2s - 2t)a_{s,t} \Bigr] \cDelta^{k-1-s-t} \left( \bigl(\cDelta^s\cu \bigr) \bigl( \cDelta^t\cv \bigr) \right) \mod Q
 \end{multline*}
 for all $\cu \in \cmE[w_1-2]$ and all $\cv \in \cmE[w_2]$.
 In particular, $\cD$ is tangential in its first factor if and only if~\eqref{eqn:when-tangential-s} holds.
 Switching the roles of $\cu$ and $\cv$ in~\eqref{eqn:tangential-1} yields
 \begin{multline*}
  \frac{1}{2}\cD(\cu \otimes Q\cv) \equiv \sum_{s=0}^k \sum_{t=0}^{k-s} k\multinom{k-1}{s}{t} \Bigl[ (2w_2 + n - 2t - 2)a_{s,t+1} \\
   + (2w_1 + 2w_2 + n - 2k - 2s - 2t)a_{s,t} \Bigr] \cDelta^{k-1-s-t} \left( \bigl( \cDelta^s\cu \bigr) \bigl( \cDelta^t\cv \bigr) \right) \mod Q
 \end{multline*}
 for all $\cu \in \cmE[w_1]$ and all $\cv \in \cmE[w_2-2]$.
 In particular, $\cD$ is tangential in its second factor if and only if~\eqref{eqn:when-tangential-t} holds.
\end{proof}

The proofs of \cref{ovsienko-redou-1d,ovsienko-redou-2d,ovsienko-redou-3d} amount to studying the solutions to the recursive system~\eqref{eqn:when-tangential}.

\begin{proof}[Proof of \cref{ovsienko-redou-1d}]
 Suppose first that $w_1,w_2 \not\in \mI_k$ and $w_1+w_2 \not\in \mO_k$.
 Then the coefficients in~\eqref{eqn:when-tangential} are nonzero for all $s,t \in \bN_0$ such that $s+t \leq k$.
 It follows that $a_{s,t}$ is uniquely determined by the choice of $a_{0,0}$.
 Our basis element takes
 \begin{equation*}
  a_{0,0} = \frac{\Gamma\bigl( w_1 + \frac{n}{2} \bigr) \Gamma\bigl( w_2 + \frac{n}{2} \bigr)}{\Gamma\bigl( w_1 + \frac{n-2k}{2} \bigr)\Gamma\bigl( w_2 + \frac{n-2k}{2} \bigr)} .
 \end{equation*}
 
 Suppose next that exactly one of $w_1 \in \mI_k$ or $w_2 \in \mI_k$ or $w_1+w_2 \in \mO_k$ holds.
 We present the case $w_1 \in \mI_k$;
 the other cases are similar.
 Let $j \in \{ 0, \dotsc, k-1 \}$ be such that $w_1 = -\frac{n-2k}{2} - j$.
 By assumption, $-w_2 + j \not\in \{ 0 , -1, \dotsc, 1-k \}$.
 Equation~\eqref{eqn:when-tangential} requires that
 \begin{align*}
  (k-j-s-1)a_{s+1,t} & = (-w_2+j+s+t)a_{s,t} , \\
  (2w_2 + n - 2t - 2)a_{s,t+1} & = 2(-w_2+j+s+t)a_{s,t} .
 \end{align*}
 It follows that $a_{s,t}$ is uniquely determined by the choice of $a_{k-j,0}$;
 moreover, $a_{s,t}=0$ if $s < k-j$.
 Our basis element takes
 \begin{align*}
  a_{k-j,0} & = \frac{\Gamma\bigl(-w_2+k\bigr)\Gamma\bigl(0\bigr)\Gamma\bigl(w_2+\frac{n}{2}\bigr)}{\Gamma\bigl(-w_2+j\bigr)\Gamma\bigl(-j\bigr)\Gamma\bigl(w_2+\frac{n-2k}{2}\bigr)} \\
  & = (-1)^jj!\frac{\Gamma\bigl(-w_2+k\bigr)\Gamma\bigl(w_2+\frac{n}{2}\bigr)}{\Gamma\bigl(-w_2+j\bigr)\Gamma\bigl(w_2+\frac{n-2k}{2}\bigr)} .
 \end{align*}

 Suppose finally that we are in one of the last three cases of \cref{ovsienko-redou-1d}.
 We present the case $w_1,w_2 \in \mI_k$ with $w_1+w_2+n \leq k$ but $w_1 + w_2 \not\in \mO_k$;
 the other cases are similar.
 Let $i,j \in \{ 0, \dotsc, k-1 \}$ be such that $w_1 = -\frac{n-2k}{2} - i$ and $w_2 = -\frac{n-2k}{2} - j$.
 By assumption, $i+j \geq k$ and $\frac{n-2k}{2} + i + j \not\in \{ 0, -1, \dotsc, 1-k \}$.
 Equation~\eqref{eqn:when-tangential} requires that
 \begin{align*}
  (k-i-s-1)a_{s+1,t} & = \left(\frac{n-2k}{2}+i+j+s+t\right)a_{s,t} , \\
  (k-j-t-1)a_{s,t+1} & = \left(\frac{n-2k}{2}+i+j+s+t\right)a_{s,t} .
 \end{align*}
 It follows that $a_{s,t}$ is uniquely determined by the choice of $a_{k-i,k-j}$;
 moreover, $a_{s,t}=0$ if $s < k-i$ or $t < k-j$.
 Our basis element takes
 \begin{equation*}
  a_{k-i,k-j} = (-1)^{i+j}i!j! \frac{\Gamma\bigl(\frac{n+2k}{2}\bigr)}{\Gamma\bigl(\frac{n-2k}{2}+i+j\bigr)} .
 \end{equation*}

 Finally, the formula~\cite{GJMS1992}*{Equation~(3.5)}
 \begin{equation}
  \label{eqn:ambient-laplacian-formula}
  \cDelta(t^w f) = t^{\gamma-2} \left( -2\rho\partial_\rho^2 + (2w + n - 2 - \rho g_\rho^{ij}g_{ij}^\prime)\partial_\rho + \Delta_{g_\rho} + \frac{w}{2}g_\rho^{ij}g_{ij}^\prime \right)f ,
 \end{equation}
 $w \in \bR$ and $f = f(x,\rho)$, for the Laplacian of the ambient metric~\eqref{eqn:straight-and-normal} implies that $D_{2k;w_1,w_2}$ is natural and depends only on $g_{ij}^{(\ell)}\rv_{\rho=0}$, $\ell \leq k-1$, and $g^{ij}g_{ij}^{(k)}\rv_{\rho=0}$, where $g_{ij}^{(\ell)} := \partial_\rho^\ell g_{ij}$.
 Since $k \leq n/2$, we conclude that $D_{2k;w_1,w_2}$ is independent of the ambiguity of $g_\rho$.
\end{proof}

\begin{proof}[Proof of \cref{ovsienko-redou-2d}]
 Suppose that $w_1,w_2 \in \mI_k$ are such that $w_1+w_2+n > k$ but $w_1 + w_2 \not\in \mO_k$;
 the other cases are similar.
 Let $i,j \in \{ 0, \dotsc, k-1 \}$ be such that $w_1 = -\frac{n-2k}{2} - i$ and $w_2 = -\frac{n-2k}{2} - j$.
 By assumption, $i+j < k$ and $\frac{n-2k}{2} + i + j \not\in \{ 0, -1, \dotsc, 1-k \}$.
 Equation~\eqref{eqn:when-tangential} requires that
 \begin{align*}
  (k-i-s-1)a_{s+1,t} & = \left(\frac{n-2k}{2}+i+j+s+t\right)a_{s,t} , \\
  (k-j-t-1)a_{s,t+1} & = \left(\frac{n-2k}{2}+i+j+s+t\right)a_{s,t} .
 \end{align*}
 It follows that $a_{s,t}$ is uniquely determined by the independent choices of $a_{k-i,0}$ and $a_{0,k-j}$;
 moreover, $a_{s,t}=0$ if $s < k-i$ and $t < k-j$.
 Our two basis elements take
 \begin{align*}
  a_{k-i,0}^{(1)} & = (-1)^i i! \frac{\Gamma\bigl(\frac{n}{2}+j\bigr)}{\Gamma\bigl(\frac{n-2k}{2}+i+j\bigr)} , & a_{0,k-j}^{(1)} & = 0 , \\
  a_{0,k-j}^{(2)} & = (-1)^i i! \frac{\Gamma\bigl(\frac{n}{2}+i\bigr)\Gamma\bigl(k-i\bigr)}{\Gamma\bigl(\frac{n-2k}{2}+i+j\bigr)\Gamma\bigl(k-j\bigr)} , & a_{k-i,0}^{(2)} & = 0 .
 \end{align*}
 
 That $D_{2k;w_1,w_2}$ is natural and well-defined follows as in the proof of \cref{ovsienko-redou-1d}.
\end{proof}

\begin{proof}[Proof of \cref{ovsienko-redou-3d}]
 It readily follows from the definitions of $\mI_k$ and $\mO_k$ that if  $w_1,w_2 \in \mI_k$ and $w_1+w_2 \in \mO_k$, then $n=2k$ and $w_1=w_2=0$.
 In this case, \eqref{eqn:when-tangential} implies that
 \begin{align*}
  (k-s-1)a_{s+1,t} & = (s+t)a_{s,t} , \\
  (k-t-1)a_{s,t+1} & = (s+t)a_{s,t} .
 \end{align*}
 It readily follows that $a_{s,t}=0$ if $(s,t) \not\in \{ (0,0), (0,k) , (k,0) \}$, and that the values of $a_{0,0}$, $a_{0,k}$, and $a_{k,0}$ may be independently chosen.
 Our three basis elements take $a_{0,0}^{(1)}=a_{0,k}^{(2)}=a_{k,0}^{(3)}=1$ and all other coefficients to be zero.
 
 That $D_{2k;w_1,w_2}$ is natural and well-defined follows as in the proof of \cref{ovsienko-redou-1d}.
\end{proof}

\section{A commutator formula on $S^n$}
\label{sec:commutator}

In this section we prove \cref{commutator} by a lengthy direct computation.
Our computations implicitly include the formula for the commutator of the Ovsienko--Redou operator $D_{2k;-\frac{n-2k}{3},-\frac{n-2k}{3}}$ with a single first spherical harmonic, which may be of independent interest.

Relative to the analogous commutator identity for the GJMS operators~\cite{Case2019fl}, the most involved step in our computation is the removal of the first-order derivatives.
The result is expressed by the following lemma:

\begin{lemma}
 \label{basic-commutator}
 Let $(\bR^{n+1,1},\cg)$ be Minkowski space.
 Given $r,s,t \in \bN_0$, denote
 \begin{equation*}
  \cD_i ( \cu \otimes \cv ) := x^i \left[ \tau\cDelta^r \left( \bigl( \mL_{\cnabla x^i}\cDelta^s\cu \bigr) \bigl( \cDelta^t\cv \bigr) \right) - x^i\cDelta^r \left( \bigl( \mL_{\cnabla\tau}\cDelta^s\cu \bigr) \bigl( \cDelta^t\cv \bigr) \right) \right] .
 \end{equation*}
 Then
 \begin{multline*}
  \sum_{i=0}^n \cD_i( \cu \otimes \cv ) \equiv \tau \cDelta^r \left( \bigl( (X-r)\cDelta^s\cu \bigr) \bigl( \cDelta^t\cv \bigr) \right) \\
   - r\tau\cDelta^{r-1} \left( \bigl(\cDelta^{s+1}\cu \bigr) \bigl( \cDelta^t\cv \bigr) - \bigl( \cDelta^s\cu \bigr) \bigl( \cDelta^{t+1}\cv \bigr) \right) \mod Q
 \end{multline*}
 for all $\cu,\cv \in C^\infty(\bR^{n+1,1})$.
\end{lemma}

\begin{proof}
 To begin, observe that if $x, y \in \{ \tau, x^0, \dotsc, x^n \}$, then
 \begin{align}
  \label{eqn:commutator-Delta-x} [ \cDelta, x] & = 2\mL_{\cnabla x} , \\
  \label{eqn:commutator-Delta-Lie} [ \cDelta , \mL_{\cnabla x} ] & = 0 , \\
  \label{eqn:commutator-Lie-Lie} [ \mL_{\cnabla x}, \mL_{\cnabla y} ] & = 0 ,
 \end{align}
 where $\mL_Y\cu := \cg(Y,\cnabla \cu)$.
 Denote $X^i := \tau\cnabla x^i - x^i \cnabla\tau$.
 A direct computation using~\eqref{eqn:commutator-Delta-x}, \eqref{eqn:commutator-Delta-Lie} and~\eqref{eqn:commutator-Lie-Lie} yields
 \begin{equation}
  \label{eqn:Li}
  \begin{split}
   \cD_i (\cu \otimes \cv) & = x^i \cDelta^r \left( \bigl( \mL_{X^i}\cDelta^s\cu \bigr) \bigl( \cDelta^t\cv \bigr) \right) \\
    & \quad - 2rx^i \cDelta^{r-1} \left( \bigl( \mL_{\cnabla x^i}\cDelta^s\cu \bigr) \bigl( \mL_{\cnabla \tau} \cDelta^t \cv \bigr) - \bigl( \mL_{\cnabla\tau} \cDelta^s\cu \bigr) \bigl( \mL_{\cnabla x^i} \cDelta^t \cv \bigr) \right) \\
   & = \cDelta^r \left( \bigl( x^i \mL_{X^i} \cDelta^s\cu \bigr) \bigl( \cDelta^t\cv \bigr) \right) \\
    & \quad - 2r \cDelta^{r-1} \Bigl( \bigl( \tau\mL_{\cnabla x^i}\mL_{\cnabla x^i} \cDelta^s\cu \bigr) \bigl( \cDelta^t\cv \bigr) - \bigl( (\mL_{x^i\cnabla x^i} + 1) \mL_{\cnabla\tau} \cDelta^s\cu \bigr) \bigl( \cDelta^t\cv \bigr) \\
    & \qquad + \bigl( \tau\mL_{\cnabla x^i}\cDelta^s\cu \bigr) \bigl( \mL_{\cnabla x^i} \cDelta^t \cv \bigr) - \bigl( \mL_{\cnabla \tau}\cDelta^s\cu \bigr) \bigl( \mL_{x^i\cnabla x^i}\cDelta^t \cv \bigr) \\
    & \qquad + \bigl( \mL_{x^i\cnabla x^i}\cDelta^s\cu \bigr) \bigl( \mL_{\cnabla\tau}\cDelta^t\cv \bigr) - \bigl( \mL_{\cnabla\tau}\cDelta^s\cu \bigr) \bigl( \mL_{x^i\cnabla x^i} \cDelta^t\cv \bigr) \Bigr) \\
   & \quad + 4r(r-1)\cDelta^{r-2} \Bigl( \bigl( \mL_{\cnabla x^i} \mL_{\cnabla x^i} \cDelta^s\cu \bigr) \bigl( \mL_{\cnabla\tau} \cDelta^t\cv \bigr) - \bigl( \mL_{\cnabla x^i} \mL_{\cnabla\tau} \cDelta^s \cu \bigr) \bigl( \mL_{\cnabla x^i} \cDelta^t\cv \bigr) \\
    & \qquad + \bigl( \mL_{\cnabla x^i}\cDelta^s\cu \bigr) \bigl( \mL_{\cnabla x^i}\mL_{\cnabla \tau} \cDelta^t\cv \bigr) - \bigl( \mL_{\cnabla\tau} \cDelta^s\cu \bigr) \bigl( \mL_{\cnabla x^i} \mL_{\cnabla x^i} \cDelta^t \cv \bigr) \Bigr) .
  \end{split}
 \end{equation}
 Next observe that
 \begin{align}
  \label{eqn:sum-xi} \sum_{i=0}^n (x^i)^2 & = Q + \tau^2 , \\
  \label{eqn:sum-Lie-xi} \sum_{i=0}^n \mL_{x^i\cnabla x^i} & = X + \tau\mL_{\cnabla\tau} , \\
  \label{eqn:sum-Lie2} \sum_{i=0}^n \mL_{\cnabla x^i}\mL_{\cnabla x^i} & = \cDelta + \mL_{\cnabla\tau}\mL_{\cnabla\tau} .
 \end{align}
 Additionally,
 \begin{align}
  \label{eqn:triple-sum} \sum_{i=0}^n \bigl( \mL_{\cnabla x^i} \mL_{\cnabla\tau} \cu \bigr) \bigl( \mL_{\cnabla x^i}\cv \bigr) & = \frac{1}{2} \cDelta \left( \cv \mL_{\cnabla\tau} \cu \right) - \frac{1}{2}\left( \cv\mL_{\cnabla\tau}\cDelta\cu + \bigl( \mL_{\cnabla\tau}\cu \bigr) \bigl( \cDelta\cv \bigr) \right) \\
   \notag & \quad + \bigl( \mL_{\cnabla\tau}\mL_{\cnabla\tau}\cu \bigr) \bigl( \mL_{\cnabla\tau}\cv \bigr) , \\
  \label{eqn:square} \sum_{i=0}^n \bigl( \mL_{\cnabla x^i}\cu \bigr) \bigl( \mL_{\cnabla x^i}\cv \bigr) & = \frac{1}{2}\cDelta\bigl( \cu\cv \bigr) - \frac{1}{2}\cv\cDelta\cu - \frac{1}{2}\cu\cDelta\cv + \bigl( \mL_{\cnabla\tau}\cu \bigr) \bigl( \mL_{\cnabla\tau} \cv \bigr)
 \end{align}
 for all $\cu, \cv \in C^\infty(\bR^{n+1,1})$.
 Set
 \begin{equation*}
  \cD := \sum_{i=0}^n \cD_i( \cu \otimes \cv ) .
 \end{equation*}
 Combining~\eqref{eqn:Li} through~\eqref{eqn:square} yields
 \begin{align*}
  \cD( \cu \otimes \cv ) & = \cDelta^r \left( \bigl( \tau X \cDelta^s\cu \bigr) \bigl( \cDelta^t \cv \bigr) - \bigl( Q\mL_{\cnabla\tau}\cDelta^s \cu \bigr) \bigl( \cDelta^t\cv \bigr) \right) \\
   & - 2r\cDelta^{r-1} \Bigl( \bigl( \tau\cDelta^{s+1}\cu \bigr) \bigl( \cDelta^t\cv \bigr) - \bigl( (X + n + 1)\mL_{\cnabla\tau}\cDelta^s\cu \bigr) \bigl( \cDelta^t\cv \bigr) \\
    & \qquad + \frac{\tau}{2}\cDelta \left( \bigl( \cDelta^s\cu \bigr) \bigl( \cDelta^t\cv \bigr) \right) - \frac{\tau}{2}\bigl( \cDelta^{s+1}\cu \bigr) \bigl( \cDelta^t\cv \bigr) - \frac{\tau}{2}\bigl( \cDelta^s\cu \bigr) \bigl( \cDelta^{t+1}\cv \bigr) \\
    & \qquad + \bigl( X\cDelta^s\cu \bigr) \bigl( \mL_{\cnabla\tau}\cDelta^t\cv \bigr) - 2\bigl( \mL_{\cnabla\tau}\cDelta^s\cu \bigr) \bigl( X\cDelta^t\cv \bigr) \Bigr) \\
   & \quad + 4r(r-1)\cDelta^{r-2} \Bigl( \bigl( \cDelta^{s+1}\cu \bigr) \bigl( \mL_{\cnabla\tau}\cDelta^t\cv \bigr) - \frac{1}{2}\cDelta \left( \bigl( \mL_{\cnabla\tau}\cDelta^s\cu \bigr) \bigl( \cDelta^t\cv \bigr) \right) \\
    & \qquad + \frac{1}{2}\bigl( \mL_{\cnabla\tau}\cDelta^{s+1}\cu \bigr) \bigl( \cDelta^t\cv \bigr) + \frac{1}{2}\bigl( \mL_{\cnabla\tau}\cDelta^s\cu \bigr) \bigl( \cDelta^{t+1}\cv \bigr) + \frac{1}{2}\cDelta \left( \bigl( \cDelta^s\cu \bigr) \bigl( \mL_{\cnabla\tau} \cDelta^t\cv \bigr) \right) \\
    & \qquad - \frac{1}{2} \bigl( \cDelta^{s+1}\cu \bigr) \bigl( \mL_{\cnabla\tau}\cDelta^t\cv \bigr) - \frac{1}{2} \bigl( \cDelta^s\cu \bigr) \bigl( \mL_{\cnabla\tau}\cDelta^{t+1}\cv \bigr) - \bigl( \mL_{\cnabla\tau}\cDelta^s\cu \bigr) \bigl( \cDelta^{t+1}\cv \bigr) \Bigr) .
 \end{align*}
 Simplifying this by combining like terms and using~\eqref{eqn:commutator-Delta-x} yields
 \begin{equation}
  \label{eqn:commutator-almost-done}
  \begin{split}
   \cD( \cu \otimes \cv) & = \tau\cDelta^r \left( \bigl( (X-r)\cDelta^s\cu \bigr) \bigl( \cDelta^t\cv \bigr) \right) - \cDelta^r \left( Q\bigl( \mL_{\cnabla\tau}\cDelta^s\cu \bigr) \bigl( \cDelta^t\cv \bigr) \right) \\
    & \quad - r\tau\cDelta^{r-1} \left( \bigl( \cDelta^{s+1}\cu \bigr) \bigl( \cDelta^t v \bigr) - \bigl( \cDelta^s\cu \bigr) \bigl( \cDelta^{t+1}\cv \bigr) \right) \\
    & \quad + 2r\cDelta^{r-1} \Bigl( \bigl( \mL_{\cnabla\tau} X\cDelta^s\cu \bigr) \bigl( \cDelta^t\cv \bigr) + \bigl( (X+n+1)\mL_{\cnabla\tau}\cDelta^s\cu \bigr) \bigl( \cDelta^t\cv \bigr) \\
     & \qquad - 2(r-1)\bigl( \mL_{\cnabla\tau}\cDelta^s\cu \bigr) \bigl( \cDelta^t\cv \bigr) + 2\bigl( \mL_{\cnabla\tau}\cDelta^s\cu \bigr) \bigl( X\cDelta^t\cv \bigr) \Bigr) .
  \end{split}
 \end{equation}
 Now recall that
 \begin{equation}
  \label{eqn:Lie-X} [ \mL_{\cnabla\tau} , X ] = \mL_{\cnabla\tau} .
 \end{equation}
 Combining~\eqref{eqn:Deltaell-Q-commutator}, \eqref{eqn:commutator-almost-done} and~\eqref{eqn:Lie-X} yields the desired conclusion.
\end{proof}

It is now straightforward to compute the commutator of $D_{2k;-\frac{n-2k}{3},-\frac{n-2k}{3}}$ with the first spherical harmonics.

\begin{proof}[Proof of \cref{commutator}]
 Define $\mC \colon \cmE\bigl[-\frac{n-2k+3}{3}\bigr] \otimes \cmE\bigl[ -\frac{n-2k}{3} \bigr] \to \cmE\bigl[ -\frac{2n+2k+3}{3} \bigr]$ by
 \begin{equation*}
  \mC( \cu \otimes \cv ) := \sum_{i=0}^n x^i \left( \tau \cD_{2k}(x^i\cu \otimes \cv) - x^i \cD_{2k}( \tau\cu \otimes \cv ) \right) .
 \end{equation*}
 \Cref{ovsienko-redou-1d} implies that $\mC$ is tangential, and hence
 \begin{equation*}
  \mC( \cu \otimes \cv) \rv_{Q=0,\tau=1} = \sum_{i=0}^n x^i [D_{2k},x^i](u \otimes v) ,
 \end{equation*}
 where $u := \cu\rv_{Q=0,\tau=1}$ and $v := \cv\rv_{Q=0,\tau=1}$.
 Direct computation using~\eqref{eqn:commutator-Delta-x} and~\eqref{eqn:commutator-Delta-Lie} yields
 \begin{align*}
  \MoveEqLeft \mC( \cu \otimes \cv ) = 2\sum_{i=0}^n \sum_{s=0}^k \sum_{t=0}^{k-s} \multinom{k}{s}{t} x^i \Bigl[ (k-s-t)a_{s,t} \mL_{X^i}\cDelta^{k-1-s-t} \left( \bigl( \cDelta^s\cu \bigr) \bigl( \cDelta^t\cv \bigr) \right) \\
  & + sa_{s,t}\left( \tau\cDelta^{k-s-t}\left( \bigl(\mL_{\nabla x^i}\cDelta^{s-1}\cu \bigr) \bigl( \cDelta^t\cv \bigr) \right) - x^i\cDelta^{k-s-t} \left( \bigl( \mL_{\nabla\tau} \cDelta^{s-1}\cu \bigr) \bigl( \cDelta^t\cv \bigr) \right) \right) \Bigr] .
 \end{align*}
 Applying \cref{basic-commutator,eqn:sum-xi,eqn:sum-Lie-xi} yields
 \begin{align*}
  \MoveEqLeft \mC( \cu \otimes \cv ) \equiv 2\tau \sum_{s=0}^k \sum_{t=0}^{k-s} \multinom{k}{s}{t} \Bigl[ (k-s-t)a_{s,t} X\cDelta^{k-1-s-t} \left( \bigl( \cDelta^s\cu \bigr) \bigl( \cDelta^t\cv \bigr) \right) \\
   & + sa_{s,t}\cDelta^{k-s-t} \left( \bigl( (X-k+s+t)\cDelta^{s-1}\cu \bigr) \bigl( \cDelta^t\cv \bigr) \right) \\
   & - (k-s-t)sa_{s,t} \cDelta^{k-1-s-t} \left( \bigl( \cDelta^s\cu \bigr) \bigl( \cDelta^t\cv \bigr) - \bigl( \cDelta^{s-1}\cu \bigr) \bigl( \cDelta^{t+1} \cv \bigr) \right) \Bigr] \mod Q .
 \end{align*}
 Evaluating the weights and re-indexing the summation yields
 \begin{align*}
  \mC( \cu \otimes \cv ) & \equiv -2\tau \sum_{s=0}^k \sum_{t=0}^{k-s} k\multinom{k-1}{s}{t} \biggl[ \Bigl( \frac{2(n+k)}{3} + s - 1 \Bigr)a_{s,t} - ta_{s+1,t-1} \\
   & \quad + \Bigl( \frac{n+k}{3} + s - t \Bigr)a_{s+1,t} \biggr] \cDelta^{k-1-s-t} \left( \bigl( \cDelta^s\cu \bigr) \bigl( \cDelta^t\cv \bigr) \right) \mod Q .
 \end{align*}
 Denote
 \begin{equation*}
  c_{s,t} := \left( \frac{2(n+k)}{3} + s - 1 \right)a_{s,t} - ta_{s+1,t-1} + \left( \frac{n+k}{3} + s - t \right)a_{s+1,t} .
 \end{equation*}
 Direct computation gives
 \begin{equation*}
  c_{s,t} = \frac{(n+2k-2)(n+k-3)}{6}\frac{ \Gamma\bigl( \frac{n-2k}{6} + s + t \bigr) \Gamma\bigl( \frac{n+4k}{6} - s - 1 \bigr) \Gamma\bigl( \frac{n+4k}{6} - t \bigr) }{ \Gamma\bigl( \frac{n-2k}{6} \bigr)^3 } .
 \end{equation*}
 We conclude that
 \begin{equation*}
  \mC( \cu \otimes \cv ) \equiv -\frac{k(n+2k-2)(n+k-3)(n-2k)}{18} \tau \cD_{2k-2; -\frac{n-2k+3}{3},-\frac{n-2k}{3}} \mod Q .
 \end{equation*}
 Restricting to $S^n$ yields the desired conclusion.
\end{proof}

\section{An intrinsic formula on Einstein manifolds}
\label{sec:einstein}

In this section we give an intrinsic formula for the Ovsienko--Redou operators of an Einstein manifold.

To begin, recall~\cite{FeffermanGraham2012} that the ambient space of an Einstein manifold is known.
This gives rise to a simple formula for the ambient Laplacian acting on suitably chosen extensions (cf.\ \cite{Matsumoto2013}*{Lemma~4.1}).

\begin{lemma}
 \label{einstein-compute}
 Let $(M^n,g)$ be an Einstein manifold with $\Ric = 2(n-1)\lambda g$ and let
 \begin{equation*}
  \bigl(\cmG , \cg \bigr) := \bigl( \bR_+ \times M \times (-\varepsilon,\varepsilon) , 2\rho \, dt^2 + 2t \, dt \, d\rho + t^2(1+\lambda\rho)^2g \bigr)
 \end{equation*}
 be its ambient space.
 Set $\tau := t(1+\lambda\rho)$.
 For any $f \in C^\infty(M)$ and $w \in \bR$, it holds that
 \begin{equation*}
  \cDelta \tau^w \cf = \tau^{w-2} \left( \Delta + 2w(n+w-1)\lambda \right) \cf ,
 \end{equation*}
 where $\cf(t,x,\rho) := f(x)$.
\end{lemma}

\begin{proof}
 One readily computes from~\eqref{eqn:ambient-laplacian-formula} that if $f=f(x,\rho)$, then
 \begin{multline*}
  \cDelta t^wf = t^{w-2} \bigl( -2\rho\partial_\rho^2 + (2w+n-2-2n(1+\lambda\rho)^{-1}\lambda\rho)\partial_\rho \\
   + \Delta_{g_\rho} + nw(1+\lambda\rho)^{-1}\lambda \bigr)f .
 \end{multline*}
 The conclusion readily follows.
\end{proof}

Applying \cref{einstein-compute} yields an intrinsic formula for the Ovsienko--Redou operators of an Einstein manifold.

\begin{proof}[Proof of \cref{einstein}]
 Let $r \in \bN_0$.
 \Cref{einstein-compute} implies that
 \begin{equation*}
  \cDelta^r \tau^{w} f = \tau^{w-2r}L_{r;w}f 
 \end{equation*}
 for any $r \in \bN$ and $w \in \bR$.
 The conclusion readily follows.
\end{proof}

\section{Conformally invariant differential operators}
\label{sec:fsa}

We conclude this article by discussing the formal self-adjointness of the Ovsienko--Redou operators and the conformally invariant operators they determine as in \cref{linear-operators}.
A key simplifying step is the following linear analogue of \cref{when-tangential}.

\begin{lemma}
 \label{linear-operator-general}
 Let $(M^n,[g])$ be a conformal manifold.
 Let $k \leq n/2$ be a positive integer and fix $\cf \in \cmE[-2\ell]$.
 Then
 \begin{equation*}
  \cD_{\cf}(\cu) := \sum_{s=0}^k \binom{k}{s}\frac{\Gamma(\ell+s)\Gamma(\ell+k-s)}{\Gamma(\ell)^2} \cDelta^{k-s}\left( \cf\cDelta^s\cu \right)
 \end{equation*}
 defines a tangential operator $\cD_{\cf} \colon \cmE\bigl[-\frac{n-2k-2\ell}{2}\bigr] \to \cmE\bigl[-\frac{n+2k+2\ell}{2}\bigr]$, and hence determines a natural conformally invariant differential operator $D_{\cf} \colon \mE\bigl[-\frac{n-2k-2\ell}{2}\bigr] \to \mE\bigl[-\frac{n+2k+2\ell}{2}\bigr]$.
\end{lemma}

\begin{remark}
 The operator $D_{\cf}$ is natural in the sense that given $g \in [g]$, one can express $D_{\cf}$ as a complete contraction of a polynomial in $\nabla^\ell\Rm$, $\nabla^\ell u$, and the Taylor coefficients of $\cf$ with respect to $\rho$.
\end{remark}

\begin{proof}
 Observe that the coefficients of $\cD_{\cf}$ satisfy~\eqref{eqn:when-tangential-s}.
 It follows from the proof of \cref{when-tangential} that $\cD_{\cf}$ is tangential.
 The properties of $D_{\cf}$ follow as in the proof of \cref{ovsienko-redou-1d}.
\end{proof}

The weights in \cref{linear-operator-general} are chosen so that the induced operators $D_{\cf}$ can be formally self-adjoint.
We adapt Fefferman and Graham's derivation~\cite{FeffermanGraham2013} of Juhl's formulas~\cite{Juhl2013} for the GJMS operators to prove formal self-adjointness when $k\leq3$.

Fix a Riemannian manifold $(M^n,g)$.
Let $(\cmG,\cg)$ be its ambient space with $\cg$ as in~\eqref{eqn:straight-and-normal}.
Set
\begin{equation*}
 w(\rho) := \left( \frac{\det g_\rho}{\det g} \right)^{1/4}
\end{equation*}
and denote $\cDelta_w := w \circ \cDelta \circ w^{-1}$.
Direct computation~\cite{FeffermanGraham2013}*{Equation~(2.4)} yields
\begin{align*}
 \cDelta_w(t^\gamma u) & = t^{\gamma-2} \left[ -2\rho\partial_\rho^2 + (2\gamma + n-2)\partial_\rho + \cmM(\rho) \right] u , \\
 \cmM(\rho) & := \delta(g_\rho^{-1}d) - \frac{[-2\rho\partial_\rho^2 + (n-2)\partial_\rho + \delta(g_\rho^{-1}d)]w(\rho)}{w(\rho)} ,
\end{align*}
where $u = u(x,\rho)$ and $\delta$ is the divergence with respect to $g=g_0$.
Note in particular that $\cmM(\rho)$ is a second-order, formally self-adjoint operator on $(M^n,g)$ for each $\rho \in \bR$.
Indeed, we may consider $\cmM(\rho)$ as a generating function
\begin{equation*}
 \cmM(\rho) = \sum_{N\geq1} \frac{1}{\bigl((N-1)!\bigr)^2}\left(-\frac{\rho}{2}\right)^{N-1}\mM_{2N} 
\end{equation*}
for a family $\{ \mM_{2N} \}_{N\in\bN}$ of second-order, formally self-adjoint operators on $(M^n,g)$.
Denote
\begin{equation}
 \label{eqn:mR}
 \mR_j := -2\rho\partial_\rho^2 + 2j\partial_\rho + \cmM(\rho) .
\end{equation}
Given $u \in C^\infty(M)$, define $\cu := t^{-\frac{n-2k-2\ell}{2}}w^{-1}u$, where $u(x,\rho) := u(x)$.
Given $\cf \in \cmE[-2\ell]$, define $f=f(x,\rho)$ by $\cf = t^{-2\ell}f$.
Since $w(0)=1$ and $\cDelta_w^k = w \circ \cDelta^k \circ w^{-1}$ for all $k \in \bN_0$, we conclude from \cref{linear-operator-general} that
\begin{equation}
 \label{eqn:linear-operator-general-nice}
 \begin{split}
 D_{\cf}(u) & = \left. \sum_{s=0}^k b_s\mR_{1-k-\ell}\dotsm\mR_{k-2s-\ell-1} \bigl( f\mR_{k+\ell-2s+1}\dotsm\mR_{k+\ell-1}u \bigr) \right|_{\rho=0} , \\
  b_s & := \binom{k}{s} \frac{\Gamma(\ell+s)\Gamma(\ell+k-s)}{\Gamma(\ell)^2} ,
 \end{split}
\end{equation}
where
\begin{align*}
 \mR_{1-k-\ell}\dotsm\mR_{k-2s-\ell-1} & := \prod_{j=1}^{k-s} \mR_{2j-k-\ell-1} , \\
 \mR_{k+\ell-2s+1}\dotsm\mR_{k+\ell-1} & := \prod_{j=1}^s \mR_{k+\ell-2s-1+2j}
\end{align*}
with the convention that the empty product equals the identity operator.

We expect that one can expand~\eqref{eqn:linear-operator-general-nice} to produce a manifestly formally self-adjoint formula for $D_{\cf}$ for all orders (cf.\ \citelist{ \cite{FeffermanGraham2013}*{Theorem~1.1} \cite{Juhl2013}*{Theorem~1.1} }).
The next three lemmas verify this expectation in the low-order cases $k \in \{ 1, 2, 3\}$.

Throughout the rest of this section we use primes to denote the evaluation of derivatives of given extensions with respect to $\rho$ at $\rho=0$;
e.g.\ $f^\prime := \partial_\rho\cf\rv_{\rho=0}$ and $f^{\prime\prime} := \partial_\rho^2\cf\rv_{\rho=0}$.

\begin{lemma}
 \label{second-order-fsa}
 Let $(M^n,g)$ be a Riemannian manifold and fix $\cf \in \cmE[-2\ell]$.
 Let $D_{\cf} \colon \mE\bigl[-\frac{n-2-2\ell}{2}\bigr] \to \mE\bigl[-\frac{n+2+2\ell}{2}\bigr]$ be as in \cref{linear-operator-general}.
 Then
 \begin{equation*}
  D_{\cf}u = \ell \bigl( f\mM_2u + \mM_2(uf) - 2\ell u f^\prime \bigr) .
 \end{equation*}
 In particular, $D_{\cf}$ is formally self-adjoint.
\end{lemma}

\begin{proof}
 Equation~\eqref{eqn:linear-operator-general-nice} asserts that
 \begin{equation*}
  D_{\cf}u = \left. \ell\bigl( \mR_{-\ell}(uf) + f\mR_\ell u) \right|_{\rho=0} .
 \end{equation*}
 Expanding this via~\eqref{eqn:mR} and recalling that $u^\prime=0$ yields the claimed formula for $D_{\cf}$.
 The final conclusion follows from the fact that $\mM_2$ is formally self-adjoint.
\end{proof}

\begin{lemma}
 \label{fourth-order-fsa}
 Let $(M^n,g)$, $n \geq 4$, be a Riemannian manifold and fix $\cf \in \cmE[-2\ell]$.
 Let $D_{\cf} \colon \mE\bigl[-\frac{n-4-2\ell}{2}\bigr] \to \mE\bigl[-\frac{n+4+2\ell}{2}\bigr]$ be as in \cref{linear-operator-general}.
 Then
 \begin{multline*}
  D_{\cf}u = \ell(\ell+1) \Bigl( f\mM_2^2u + \mM_2^2(uf) + \frac{2\ell}{\ell+1} \mM_2(f\mM_2u) \\
    + (\ell+1)\bigl(f \mM_4u + \mM_4(uf) \bigr) - 4\ell\bigl( f^\prime \mM_2u + \mM_2(uf^\prime) \bigr) + 4\ell(\ell+1)uf^{\prime\prime} \Bigr) .
 \end{multline*}
 In particular, $D_{\cf}$ is formally self-adjoint.
\end{lemma}

\begin{remark}
 If $\ell=-1$, then the conclusion of \cref{fourth-order-fsa} is that
 \begin{equation*}
  D_{\cf}u = 2\mM_2(f\mM_2u) .
 \end{equation*}
\end{remark}

\begin{proof}
 Equation~\eqref{eqn:linear-operator-general-nice} asserts that
 \begin{equation}
  \label{eqn:pre-fourth-order-fsa}
  D_{\cf}u = \ell(\ell+1)\mR_{-1-\ell}\mR_{1-\ell}(uf) + 2\ell^2\mR_{-1-\ell}(f\mR_{\ell+1}u) + \ell(\ell+1)f\mR_{\ell-1}\mR_{\ell+1}u ,
 \end{equation}
 where the right-hand side is evaluated at $\rho=0$.
 Recall that $u^\prime=0$.
 Direct computation yields
 \begin{align*}
  f\mR_{\ell-1}\mR_{\ell+1}u & = -(\ell-1)f\mM_4u + f\mM_2^2u , \\
  \mR_{-1-\ell} (f\mR_{\ell+1}u) & = -2(\ell+1)f^\prime\mM_2u + (\ell+1)f\mM_4u + \mM_2(f\mM_2u) , \\
  \mR_{-1-\ell}\mR_{1-\ell}(uf) & = 4\ell(\ell+1)uf^{\prime\prime} - 4\ell\mM_2(uf^\prime) + (\ell+1)\mM_4(uf) + \mM_2^2(uf)
 \end{align*}
 at $\rho=0$.
 Inserting these into~\eqref{eqn:pre-fourth-order-fsa} yields the claimed formula for $D_{\cf}$.
 The final conclusion follows from the fact that $\mM_2$ and $\mM_4$ are formally self-adjoint.
\end{proof}

\begin{lemma}
 \label{sixth-order-fsa}
 Let $(M^n,g)$, $n \geq 6$, be a Riemannian manifold and fix $\cf \in \cmE[-2\ell]$.
 Let $D_{\cf} \colon \mE\bigl[-\frac{n-6-2\ell}{2}\bigr] \to \mE\bigl[-\frac{n+6+2\ell}{2}\bigr]$ be as in \cref{linear-operator-general}.
 Then
 \begin{align*}
  D_{\cf}u & = \ell(\ell+1)(\ell+2) \Bigl( f\mM_2^3u + \mM_2^3(uf) + \frac{3\ell}{\ell+2}\bigl( \mM_2(f\mM_2^2u) + \mM_2^2(f\mM_2u) \bigr) \\
   & \quad + 3\ell\bigl( \mM_2(f\mM_4u) + \mM_4(f\mM_2u) \bigr) + 2(\ell+1)\bigl( f\mM_4\mM_2u + \mM_2\mM_4(uf) \bigr) \\
   & \quad + (\ell+2)\bigl( f\mM_2\mM_4u + \mM_4\mM_2(uf)\bigr) + \frac{(\ell+1)(\ell+2)}{2}\bigl( f\mM_6u + \mM_6(uf)\bigr) \\
   & \quad - \frac{12\ell(\ell+1)}{\ell+2}\mM_2(f^\prime\mM_2u) - 6\ell\bigl( f^\prime\mM_2^2u + \mM_2^2(uf^\prime \bigr) \\
   & \quad - 6\ell(\ell+2)\bigl( f^\prime\mM_4 + \mM_4(f^\prime u) \bigr) + 12\ell(\ell+1)\bigl( f^{\prime\prime}\mM_2u + \mM_2(uf^{\prime\prime})\bigr) \\
   & \quad - 8\ell(\ell+1)(\ell+2)uf^{\prime\prime\prime} \Bigr)
\end{align*}
 In particular, $D_{\cf}$ is formally self-adjoint.
\end{lemma}

\begin{remark}
 If $\ell=-2$, then the conclusion of \cref{sixth-order-fsa} is that
 \begin{equation*}
  D_{\cf}u = -12\bigl( \mM_2(f\mM_2^2u) + \mM_2^2 (f\mM_2u) \bigr) - 48\bigl(\mM_2(f^\prime\mM_2u \bigr) .
 \end{equation*}
\end{remark}

\begin{proof}
 Equation~\eqref{eqn:linear-operator-general-nice} asserts that
 \begin{equation}
  \label{eqn:pre-sixth-order-fsa}
  \begin{split}
   D_{\cf}u & = \ell(\ell+1)(\ell+2)\mR_{-2-\ell}\mR_{-\ell}\mR_{2-\ell}(uf) \\
    & \quad + 3\ell^2(\ell+1)\mR_{-2-\ell}\mR_{-\ell}\bigl( f\mR_{\ell+2}u \bigr) + 3\ell^2(\ell+1)\mR_{-2-\ell}\bigl( f \mR_{\ell}\mR_{\ell+2}u \bigr) \\
    & \quad + \ell(\ell+1)(\ell+2)f\mR_{\ell-2}\mR_{\ell}\mR_{\ell+2}u ,
  \end{split}
 \end{equation}
 where the right-hand side is evaluated at $\rho=0$.
 Recall that $u^\prime=0$.
 Direct computation yields
 \begin{multline*}
  f\mR_{\ell-2}\mR_{\ell}\mR_{\ell+2}u = \frac{(\ell-1)(\ell-2)}{2}f\mM_6u - 2(\ell-1)f\mM_2\mM_4u \\
   - (\ell-2)f\mM_4\mM_2u + f\mM_2^3u
 \end{multline*}
 and
 \begin{multline*}
  \mR_{-2-\ell}\bigl( f\mR_{\ell}\mR_{\ell+2}u \bigr) = 2\ell(\ell+2)f^\prime\mM_4u - 2(\ell+2)f^\prime \mM_2^2u - \frac{(\ell-1)(\ell+2)}{2}f\mM_6u \\
   + (\ell+2)f\bigl( \mM_2\mM_4u + \mM_4\mM_2u \bigr) - \ell\mM_2\bigl( f\mM_4u \bigr) + \mM_2\bigl( f\mM_2^2u \bigr)
 \end{multline*}
 and
 \begin{align*}
  \MoveEqLeft \mR_{-2-\ell}\mR_{-\ell}\bigl( f\mR_{\ell+2}u \bigr) = 4(\ell+1)(\ell+2)f^{\prime\prime}\mM_2u - 4(\ell+1)(\ell+2)f^\prime\mM_4u \\
   & \quad - 4(\ell+1)\mM_2\bigl( f^\prime\mM_2u \bigr) + \frac{(\ell+1)(\ell+2)}{2}f\mM_6u \\
   & \quad + (\ell+2)\mM_4\bigl( f\mM_2u \bigr) + 2(\ell+1)\mM_2\bigl( f\mM_4u \bigr) + \mM_2^2\bigl( f\mM_2u \bigr)
 \end{align*}
 and
 \begin{align*}
  \MoveEqLeft \mR_{-2-\ell}\mR_{-\ell}\mR_{2-\ell}(uf) = -8\ell(\ell+1)(\ell+2)uf^{\prime\prime\prime} + 12\ell(\ell+1)\mM_2\bigl( uf^{\prime\prime} \bigr) - 6\ell\mM_2^2\bigl(uf^\prime\bigr) \\
   & \quad - 6\ell(\ell+2)\mM_4\bigl(uf^\prime\bigr) + \frac{(\ell+1)(\ell+2)}{2}\mM_6(uf) \\
   & \quad + 2(\ell+1)\mM_2\mM_4(uf) + (\ell+2)\mM_4\mM_2(uf) + \mM_2^3(uf)
 \end{align*}
 at $\rho=0$.
 Inserting these identities into~\eqref{eqn:pre-sixth-order-fsa} yields the claimed formula for $D_{\cf}$.
 The final conclusion follows from the fact that $\mM_2$, $\mM_4$, and $\mM_6$ are formally self-adjoint.
\end{proof}

Combining \cref{second-order-fsa,fourth-order-fsa,sixth-order-fsa} yields the formal self-adjointness of a family of operators relevant both to the Ovsienko--Redou operators $D_{2k;-\frac{n-2k}{3},-\frac{n-2k}{3}}$ and the differential operators of \cref{linear-operators}.

\begin{corollary}
 \label{general-fsa}
 Let $(M^n,g)$ be a Riemannian manifold and let $k \leq \min \{ 3, n/2 \}$ be a positive integer.
 Given $\cf \in \cmE[-2\ell]$, let $D \colon \mE\bigl[-\frac{n-2k-2\ell}{2}\bigr] \to \mE\bigl[-\frac{n+2k+2\ell}{2}\bigr]$ be the natural conformally invariant differential operator determined by
 \begin{equation}
  \label{eqn:general-fsa}
  \cD(\cu) := \cD_{2k;-2\ell,-\frac{n-2k-2\ell}{2}}(\cf \otimes \cu)
 \end{equation}
 for $\cu \in \cmE\bigl[ -\frac{n-2k-2\ell}{2} \bigr]$.
 Then $D$ is formally self-adjoint.
\end{corollary}

\begin{proof}
 Direct computation yields
 \begin{equation*}
  \cD(\cu) = \sum_{s=0}^k \binom{k}{s}\frac{\Gamma\bigl(\frac{n}{2}-2\ell-s\bigr)\Gamma\bigl(\ell+s\bigr)^2}{\Gamma\bigl(\frac{n}{2}-2\ell-k\bigr)\Gamma\bigl(\ell\bigr)^2} \cD_{\cDelta^s\cf}(\cu) .
 \end{equation*}
 We conclude from \cref{linear-operator-general} that $\cD$ is a linear combination of tangential operators with common domain $\cmE\bigl[-\frac{n-2k-2\ell}{2}\bigr]$.
 Since $k \leq 3$, we conclude from \cref{second-order-fsa,fourth-order-fsa,sixth-order-fsa} that the induced operator $D$ is a linear combination of formally self-adjoint operators.
\end{proof}

We conclude with the proofs that the Ovsienko--Redou operators and that differential operators of \cref{linear-operators} are formally self-adjoint if they involve at most six ambient derivatives of the input functions.

\begin{proof}[Proof of \cref{or-fsa}]
 Denote $D_{2k} := D_{2k;-\frac{n-2k}{3},-\frac{n-2k}{3}}$.
 The fact that $a_{s,t}=a_{t,s}$ for this choice of weights implies that $D_{2k}$ is symmetric in its arguments.
 Applying \cref{general-fsa} with $\cf = \cu$ implies that the map $v \mapsto D_{2k}(u \otimes v)$ is formally self-adjoint for all $u \in \mE\bigl[-\frac{n-2k}{3}\bigr]$.
 Combining these observations yields the desired result.
\end{proof}

\begin{proof}[Proof of \cref{linear-operators}]
 This follows immediately from \cref{general-fsa} with $\cf = \cI$.
\end{proof}

\begin{remark}
 The conjectured formal self-adjointness of the introduction follows if the operators of \cref{linear-operator-general} are formally self-adjoint for all $k \in \bN$.
 This is readily seen from the proofs of \cref{general-fsa,or-fsa,linear-operators}.
\end{remark}

 \section*{Acknowledgements}
JSC was supported by the Simons Foundation (Grant \#524601).
He also thanks the University of Washington for providing a productive research environment while part of this research was carried out.
WY was supported by NSFC (Grant No.\ 12071489, No.\ 12025109).

\bibliography{bib}
\end{document}